\documentclass[a4,12pt]{amsart}
\usepackage{amssymb,amstext,amsmath,amscd,amsthm,amsfonts,enumerate,latexsym}
\usepackage{color}
\usepackage[dvipdfmx]{graphicx}
\usepackage[all]{xy}
\theoremstyle{plain}
\newtheorem{thm}{Theorem}[section]

\newtheorem*{thm*}{Theorem}
\newtheorem*{cor*}{Corollary}

\newtheorem{prop}[thm]{Proposition}

\newtheorem{lemma}[thm]{Lemma}
\newtheorem{lem}[thm]{Lemma}
\newtheorem{cor}[thm]{Corollary}

\newtheorem{claim}{Claim}
\newtheorem*{claim*}{Claim}
\newtheorem*{acknowledgement}{Acknowledgement}

\theoremstyle{definition}
\newtheorem{defn}[thm]{Definition}

\newtheorem{ex}[thm]{Example}
\newtheorem{Example}[thm]{Example}

\theoremstyle{remark}

\numberwithin{equation}{thm}

\def\Ext{\operatorname{Ext}}

\def\E{\operatorname{E}}

\def\Hom{\operatorname{Hom}}

\def\End{\mathrm{End}}

\def\Coker{\mathrm{Coker}}

\def\m{\mathfrak m}
\def\n{\mathfrak n}

\def\N{\Bbb N}

\newcommand{\rme}{\mathrm{e}}

\newcommand{\rmr}{\mathrm{r}}

\newcommand{\rmE}{\mathrm{E}}

\newcommand{\rmJ}{\mathrm{J}}

\newcommand{\rmQ}{\mathrm{Q}}

\newcommand{\calF}{\mathcal{F}}
\newcommand{\calG}{\mathcal{G}}

\newcommand{\calX}{\mathcal{X}}
\newcommand{\calY}{\mathcal{Y}}
\newcommand{\calZ}{\mathcal{Z}}

\newcommand{\fka}{\mathfrak{a}}

\newcommand{\fkm}{\mathfrak{m}}

\newcommand{\fkp}{\mathfrak{p}}

\newcommand{\mapright}[1]{%
\smash{\mathop{%
\hbox to 1cm{\rightarrowfill}}\limits^{#1}}}

\newcommand{\mapleft}[1]{%
\smash{\mathop{%
\hbox to 1cm{\leftarrowfill}}\limits_{#1}}}

\def\Spec{\operatorname{Spec}}
\def\Max{\operatorname{Max}}

\def\Ass{\operatorname{Ass}}

\def\height{\mathrm{ht}}
\def\depth{\operatorname{depth}}
\def\grade{\mathrm{grade}}

\def\ol{\overline}

\def\N0{\mathbb{N}_{0}}

\title[The Gorenstein property and correspondences]{Correspondence between trace ideals and birational extensions  with application to the analysis of 
the Gorenstein property of rings}

\author{Shiro Goto}
\address{Department of Mathematics, School of Science and Technology, Meiji University, 1-1-1 Higashi-mita, Tama-ku, Kawasaki 214-8571, Japan}
\email{shirogoto@gmail.com}

\author{Ryotaro Isobe}
\address{Department of Mathematics and Informatics, Graduate School of Science and Technology, Chiba University, Yayoi-cho 1-33, Inage-ku, Chiba, 263-8522, Japan}
\email{axna4902@chiba-u.jp}

\author{Shinya Kumashiro}
\address{Department of Mathematics and Informatics, Graduate School of Science and Technology, Chiba University, Yayoi-cho 1-33, Inage-ku, Chiba, 263-8522, Japan}
\email{axwa4903@chiba-u.jp}

\thanks{The first author was partially supported by the JSPS Grant-in-Aid for Scientific Research (C) 16K05112. The second and third authors were partially supported by Birateral Programs (Joint Research) of JSPS and International Research Supporting Programs of Meiji University.}

\begin{document}
\maketitle

\setlength{\baselineskip}{18pt}

\begin{abstract}
Over an arbitrary commutative ring, correspondences among three sets, the set of trace ideals, the set of stable ideals, and the set of birational extensions of the base ring, are studied. The correspondences are well-behaved, if the base ring is a Gorenstein ring of dimension one. It is shown that with one extremal exception, the surjectivity of one of the correspondences characterizes the Gorenstein property of the base ring, provided it is a Cohen-Macaulay local ring of dimension one. Over a commutative Noetherian ring, a characterization of modules in which every submodule is a trace module is given. The notion of anti-stable rings is introduced, exploring their basic properties. 
\end{abstract}

{\footnotesize \tableofcontents}


\section{Introduction}\label{sec1}
This paper aims to explore the structure of (not necessarily Noetherian) commutative rings in connection with their trace ideals. Let $R$ be a commutative ring. For $R$-modules $M$ and $X$, let $$\tau_{M, X}: \Hom_R(M, X) \otimes_RM \to X$$ denote the $R$-linear map defined by $\tau_{M, X} (f \otimes m) =f(m)$ for all $f \in \Hom_R(M,X)$ and $m \in M$. We set $\tau_X(M) = \operatorname{Im} \tau_{M,X}$. Then, $\tau_X(M)$ is an $R$-submodule of $X$, and we say that an $R$-submodule $Y$ of $X$ is a {\em trace module} in $X$, if $Y =\tau_X(M)$ for some $R$-module $M$.  When $X=R$, we call trace modules in $R$, simply, {\em trace ideals} in $R$. There is a recent movement in the theory of trace ideals, raised by  of H. Lindo and N. Pande \cite{L, L2, LP}. Besides, J. Herzog, T. Hibi, and D. I. Stamate \cite{HHS} studied the traces of canonical modules, and gave interesting results. We explain below our motivation for the present researches and how this paper is organized, claiming the main results in it.

The present researches are strongly inspired by \cite{L, L2, LP}. In \cite{L2} Lindo asked when every ideal of a given ring $R$ is a trace ideal in it, and noted that this is the case when $R$ is a self-injective ring. Subsequently, Lindo and Pande \cite{LP} proved that the converse is also true if $R$ is a Noetherian local ring. Our researches have started from the following complete answer to their prediction, which we shall prove in Section 4.

\begin{thm}[Theorem \ref{1.3}]
Suppose that $R$ is a Noetherian ring and let $X$ be an $R$-module.
Then the following conditions are equivalent.
\begin{enumerate}[{\rm (1)}]
\item Every $R$-submodule of $X$ is a trace module in $X$.
\item Every cyclic $R$-submodule of $X$ is a trace module in $X$. 
\item There is an embedding $$0 \to X \to \bigoplus_{\m \in \Max R} \rmE_R (R/\m) $$
of $R$-modules, where for each $\m \in \Max R$, $\rmE_R(R/\m)$ denotes the injective envelope of the cyclic $R$-module $R/\m$.
\end{enumerate}
\end{thm}

However, the main activity in the present paper is focused on the study of the structure of the set of regular trace ideals in $R$. Let $I$ be an ideal of a commutative ring $R$ and suppose that $I$ is {\em regular}, that is $I$ contains a non-zerodivisor of $R$. Then, as is essentially shown by \cite[Lemma 2.3]{L2}, $I$ is a trace ideal in $R$ if and only if $R:I= I:I$, where the colon is considered inside of the total ring $\rmQ(R)$ of fractions of $R$. We denote by $\calX_R$ the set of regular trace ideals in $R$, and explore the structure of $\calX_R$ in connection with the structure of $\calY_R$, where $\calY_R$ denotes the set of birational extensions $A$ of $R$ such that $aA \subseteq R$ for some non-zerodivisor $a$ of $R$. We also consider the set $\calZ_R$ of regular ideals $I$ of $R$ such that $I^2=aI$ for some $a \in I$. We then have the following natural correspondences 
$$\xi : \calZ_R \to \calY_R, \ \ \xi(I)= I:I,$$
$$\eta : \calY_R \to \calX_R, \  \ \eta(A)= R:A,$$
$$\rho : \calX_R \to \calY_R,\ \  \rho(I)= I:I$$
among these sets. The basic framework is the following.

\begin{prop}[Proposition \ref{3.9}, Lemma \ref{3.6} (1)]
 The correspondence $\xi : \calZ_R \to \calY_R$ is surjective, and the following conditions are equivalent.
\begin{enumerate}[{\rm (1)}]
\item $\rho : \calX_R \to \calY_R$ is surjective.
\item $\eta : \calY_R \to \calX_R$ is injective.
\item $A=R:(R:A)$ for every $A \in \calY_R$.
\end{enumerate}
\end{prop}

Our strategy is to make use of these correspondences in order to analyze the structure of commutative rings $R$ which are not necessarily Noetherian (see, e.g., \cite{GI}). This approach is partially inspired by and originated in \cite{GIW}, where certain specific ideals (called {\em good ideals})  in Gorenstein local rings are closely studied. Similarly, as in \cite{GIW} and as is shown later in Sections 2 and 3, the above correspondences behave very well, especially in the case where $R$ is a Gorenstein ring of dimension one. We actually have $\eta \circ \rho = 1_{\calX_R}$ and $\rho \circ \eta = 1_{\calY_R}$ in the case (Lemma \ref{3.6}). Nevertheless, being different from \cite{GIW}, our present interest is in the question of when the correspondence $\rho : \calX_R \to \calY_R$ is bijective. As is shown in Section 2 (Example \ref{3.7}), in general there is no hope for the surjectivity of $\rho$ in the case where $\dim R \ge 2$, even if $R$ is a Noetherian integral domain of dimension two. On the other hand, with very specific, so to speak extremal exceptions (Proposition \ref{2.7.5}), the surjectivity of $\rho$ guarantees the Gorenstein property of $R$, provided $R$ is a Cohen-Macaulay local ring of dimension one. In fact, we will prove in Section 5 the following, in which let us refer to \cite{GMP} for the notion of almost Gorenstein local ring.

\begin{thm}[Theorem \ref{1.2}]
Let $(R, \fkm)$ be a Cohen-Macaulay local ring of dimension one. Let $B=\fkm:\fkm$ and let $\rmJ(B)$ denote the Jacobson radical of $B$. Then the following assertions are equivalent.
\begin{enumerate}[{\rm (1)}]
\item $\rho : \calX_R \to \calY_R$ is  bijective.
\item $\rho : \calX_R \to \calY_R$  is surjective.
\item Either $R$ is a Gorenstein ring, or $R$ satisfies the following two conditions.
\begin{enumerate}[{\rm (i)}]
\item $B$ is a $\operatorname{DVR}$ and $\rmJ(B)=\fkm$.
\item There is no proper intermediate field between $R/\fkm$ and $B/\rmJ(B)$.
\end{enumerate} 
\end{enumerate}
When this is the case, $R$ is an almost Gorenstein local ring in the sense of \cite{GMP}.
\end{thm}

\noindent
Therefore, $\rho$ is surjective if and only if $R$ is a Gorenstein ring, provided $R$ is the semigroup ring of a numerical semigroup over a field.


In Section 6, we introduce the notion of anti-stable and strongly anti-stable rings. We say that a commutative ring $R$ is {\em anti-stable} (resp. {\em strongly anti-stable}), if $\Hom_R(I,R)$ is an invertible module over the ring $\End_RI$ (resp. $\Hom_R(I,R) \cong \End_RI$ as an $\End_RI$-module), for every {\em regular} ideal $I$ of $R$. The purpose of Section 6 is to provide some basic properties of anti-stable rings and strongly anti-stable rings, mainly in dimension one.

Here, let us remind the reader that $R$ is said to be a {\em stable} ring, if every ideal $I$ of $R$ is {\em stable}, that is $I$ is projective over $\End_RI$ (\cite{SV}). The notion of stable ideals and rings is originated in the famous articles \cite{B} and \cite{Li} of H. Bass and J. Lipman, respectively, and there are known many deep results about them (\cite{SV}). Our definition of anti-stable rings is, of course, different from that of stable rings. It requires the projectivity of the dual module $\Hom_R(I,R)$ of $I$, only for regular ideals $I$ of $R$, claiming nothing about the projectivity of $I$ itself. Nevertheless, with some additional conditions in dimension one, $R$ is also a stable ring, once it is anti-stable, as we shall show in the following.

\begin{thm}[Theorem \ref{final}]
Let $R$ be a Cohen-Macaulay ring with $\dim R_M = 1$ for every $M \in \Max R$. If $R$ is an anti-stable ring, then $R$ is a stable ring. 
\end{thm}

The results of Section 6 are obtained as applications of the observations developed in Sections 2, 3, and 5. One can also find, in the forthcoming paper \cite{GI}, further developments of the theory of anti-stable rings of higher dimension.

Similarly as \cite{KT2}, our research is motivated by the works \cite{L, L2, LP} of Lindo and Pande, so that the topics of Section 6 are similar to those of \cite{KT2}, but these two researches were done with entire independence of each other. In \cite{LP}, Lindo and Pande posed a problem what kind of properties a Noetherian ring $R$ enjoys, if {\em every} ideal of $R$ is isomorphic to a trace ideal in it. In \cite{KT2}, T. Kobayashi and R. Takahashi have given  complete answers to the problem. We were also interested in the problem, and thereafter, came to the notion of anti-stable ring. If the ideal $I$ considered is {\em regular}, the condition (C) that $I$ is isomorphic to a trace ideal is equivalent to saying that $\Hom_R(I,R)\cong \End_RI$ as an $\End_RI$-module (Lemma \ref{5.0}). Therefore, if we restrict our attention, say  on integral domains $R$, the condition that every regular ideal satifies condition (C) is equivalent to saying that $R$ is a strongly anti-stable ring. However, in general, these two conditions are apparently different (e.g., consider the case where every non-zerodivisor of the ring is invertible in it, and see \cite[Theorem 3.2]{KT2}). It must be necessary, and might have some significance, to start a basic theory of anti-stable and strongly anti-stable rings in our context, with a different viewpoint from \cite{KT2}, which we have performed in Section 6.


In what follows, unless otherwise specified, $R$ denotes a commutative ring. Let $\rmQ(R)$ be the total ring of fractions of $R$. For $R$-submodules $X$ and $Y$ of $\rmQ(R)$, let $$X:Y = \{a \in \rmQ(R) \mid aY \subseteq X\}.$$
If we consider ideals $I,J$ of $R$, we set $I:_RJ=\{a \in R \mid aJ \subseteq I\}$; hence $$I:_RJ = (I:J) \cap R.$$ When $(R, \m)$ is a Noetherian local ring of dimension $d$, for each finitely generated $R$-module $M$, let $\mu_R(M)$ (resp. $\ell_R(M)$) denote the number of elements in a minimal system of generators (resp. the length) of $M$. We denote by $$\rme(M) = \lim_{n \to \infty}d!{\cdot}\frac{\ell_R(M/\m^{n+1}M)}{n^d}$$ the multiplicity of $M$. Let $\rmr(R) = \ell_R(\Ext_R^d(R/\m,R))$ stand for the Cohen-Macaulay type of $R$, where we assume the local ring $R$ is Cohen-Macaulay.


\section{Correspondence between trace ideals and birational extensions of the base ring}
Let $R$ be a commutative ring and let $M, X$ be $R$-modules. We denote by $\tau_{M,X}: \Hom_R(M,X) \otimes_RM \to X$ the $R$-linear map such that $\tau_{M,X}(f \otimes m)=f(m)$ for all $f \in \Hom_R(M,X)$ and $m \in M$. Let $\tau_X(M) = \operatorname{Im} \tau_{M,X}$. Then, $\tau_X(M)$ is an $R$-submodule of $X$, and we say that an $R$-submodule $Y$ of $X$ is a trace module in $X$, if $Y =\tau_X(M)$ for some $R$-module $M$. When $X = R$, we simply say that $Y$ is a trace ideal in $R$. With this notation we have the following.

\begin{prop}[{\cite[Lemma 2.3]{L2}}]\label{2.3}
For an  $R$-submodule $Y$ of $X$, the following conditions are equivalent.
\begin{enumerate}[{\rm (1)}]
\item $Y$ is a trace module in $X$.
\item $Y=\tau_X (Y)$. 
\item The embedding $\iota: Y \to X$ induces the isomorphism $\iota_* : \Hom_R(Y, Y) \to \Hom_R(Y, X)$ of $R$-modules. 
\item $f(Y) \subseteq Y$ for all $f \in \Hom_R(Y, X)$.
\end{enumerate}
\end{prop}

We denote by $W$ the set of non-zerodivisors of $R$. Let $\calF_R$ be the set of {\em regular} ideals of $R$, that is the ideals $I$ of $R$ with $I \cap W \ne \emptyset$. We then have the following, characterizing trace ideals.

\begin{cor}\label{2.5}
Let $I \in \calF_R$. Then the following conditions are equivalent.
\begin{enumerate}[{\rm (1)}]
\item $I$ is a trace ideal in $R$.
\item $I=(R:I)I$. 
\item $I:I=R:I$. 
\end{enumerate}
\end{cor}

\begin{proof}
Since $I \cap W \ne \emptyset$, we have natural identifications
$R:I = \Hom_R(I,R)$ and $I:I = \Hom_R(I,I)$, so that  the equivalence of conditions (1) and (3) follows from Proposition \ref{2.3}. Suppose that $I=(R:I)I$. Then $R:I \subseteq I:I$, whence $R:I = I:I$. Conversely, if $I : I = R:I$, then $(R:I)I = (I:I)I \subseteq I$, while $I \subseteq (R:I)I$, since $1 \in R:I$. Thus $(R:I)I= I$.
\end{proof}

We now  consider the following sets:
\begin{align*}
\calX_R &=\left\{ I \in \calF_R \mid \text{$I$ is a trace ideal in $R$} \right\},\\
\calY _R&=\left\{ A \mid \text{$R \subseteq A \subseteq \rmQ(R)$, $A$ is a subring of $\rmQ(R)$ such that $aA \subseteq R$ for some $a \in W$} \right\}, \\
\calZ_R &= \left\{ I  \in \calF_R \mid \text{$I^2=aI$ for some $a \in I$} \right\}.
\end{align*} 

\noindent
If $R$ is a Noetherian ring, then $\calY_R$ is the set of birational finite extensions of $R$. In what follows, we shall clarify the relationship among these sets. We begin with the following.

\begin{prop}\label{3.0} The following assertions hold true.
\begin{enumerate}[{\rm (1)}]
\item Let $X$ be an $R$-submodule of $\rmQ(R)$ and set $Y = R:X$. Then $Y= R: (R:Y)$. 
\item Let $I \in \calZ_R$ and assume that $I^2 = aI$ with $a \in I$. Then $a \in W$ and  $I:I = a^{-1}I$.
\end{enumerate}
\end{prop}

\begin{proof}
(1) Since $X \subseteq R:Y$, $Y = R:X \supseteq R:(R:Y)$, so that $Y= R:(R:Y)$.

(2) We have $a \in W$, because $I \in \calF_R$. Since $a \in I$, $I:I \subseteq a^{-1}I$, while $a^{-1}I \subseteq I : I$, because $a^{-1}I{\cdot}I = a^{-1}I^2 = a^{-1}(aI)=I$. Hence $I:I = a^{-1}I$.
\end{proof}

\begin{lemma}\label{3.2}
The following assertions hold true.
\begin{enumerate}[{\rm (1)}]
\item Let $I \in \calX_R$ and $a \in I \cap W$. We set $J = (a):_RI$. Then, $J \subseteq I$ and $J^2 = aJ$, so that $J \in \calZ_R$.
\item Let $I \in \calZ_R$ and write $I^2 = aI$ with $a \in I$. We set $J = (a):_RI$. Then, $I \subseteq J$ and $J \in \calX_R$.
\end{enumerate}
\end{lemma}

\begin{proof}
(1) We set $A = I:I$. Then, $A = R:I$  by Corollary \ref{2.5}. Hence, $J=(a):_RI = (a):I=a(R:I) = aA$, where the second equality follows from the fact that $a \in I \cap W$. Therefore, $J^2 = aJ$ and $J = a(I:I) \subseteq I$.

(2) Notice that $J = (a):I = a(R:I)$. Let $A =I:I$. Then, $I = aA$, since $A = a^{-1}I$ by Proposition \ref{3.0} (2), so that $R:I = R :aA = a^{-1}(R:A)$. Therefore, $J = R:A$, whence $$J:J= (R:A):(R:A) = R:A(R:A)= R:(R:A)=R:J.$$ Thus, $J \in \calX_R$. 
\end{proof}

Let $I \in \calF_R$. We say that $I$ is a {\em good ideal} of $R$, if $I^2=aI$ and $I=(a):_RI$ for some $a \in I$ (cf. \cite{GIW}). Let $\calG_R$ denote the set of  good ideals in $R$. We then have the following,  characterizing good ideals.

\begin{prop}\label{3.3}
$\calX_R \cap \calZ_R= \calG_R=\{I \in \calX_R  \mid (a):_RI \in \calX_R~\text{for ~some}~a \in I \cap W \}$.
\end{prop}

\begin{proof}
Let $I \in \calX_R \cap \calZ_R$ and set $A=I:I$. We write $I^2=aI$ with $a \in I$. Then, since  $I=aA$ and $A = R:I$ (see Proposition \ref{3.0} and Corollary \ref{2.5}), $(a):_RI= (a):I=a(R:I) = aA=I$,  so that $I$ is a good ideal of $R$. Conversely, suppose that $I$ is a good ideal of $R$ and assume that $I^2=aI$ and $I=(a):_RI$ with $a \in I$. Then $I \in \calZ_R$, while $(a):_RI \in \calX_R$ by Lemma \ref{3.2} (2). Hence $I \in \calX_R \cap \calZ_R$.

Assume that $I \in \calX_R$ and that $(a):_RI \in \calX_R$ for some $a \in I \cap W$. We set $J =(a):_RI$.  Then, $J^2 = a J$ and $J \subseteq I$, by Lemma \ref{3.2} (1). For the same reason, we get $(a):_R J \subseteq J$, because $J \in \calX_R$ and $a \in J$. Therefore, $I \subseteq (a):_RJ \subseteq J \subseteq I$; hence $I =J$. Thus, $I^2=aI$ and $I = (a):_RI$, that is $I \in \calG_R$.
\end{proof}

Let us consider three correspondences
$$\xi : \calZ_R \to \calY_R, \ \ \xi(I)= I:I,$$
$$\eta : \calY_R \to \calX_R, \  \ \eta(A)= R:A,$$
$$\rho : \calX_R \to \calY_R,\ \  \rho(I)= I:I.$$
Here, we briefly confirm the well-definedness of $\eta$. Let $A \in \calY_R$ and set $I = R:A$. Since $I$ is an ideal of $A$, we get $I:I= (R:A):I = R:AI= R:I.$ Therefore, $I \in \calX_R$, which shows $\eta$ is well-defined.

With this notation, we have the following, which plays a key role in this paper.

\begin{lemma}\label{3.6} The following assertions hold true.
\begin{enumerate}[{\rm (1)}]
\item The correspondence $\xi : \calZ_R \to \calY_R$ is surjective. For each $I,J \in \calZ_R$, $\xi(I) = \xi(J)$ if and only if $I \cong J$ as an $R$-module, so that $\calY_R = \calZ_R/\cong$, that is the set of the isomorphism classes in $\calZ_R$.
\item $\eta(\rho(I)) = R:(R:I)$ for every $I \in \calX_R$. 
\item $\rho(\eta(A)) = R:(R:A)$ for every $A \in \calY_R$.
\end{enumerate}
Consequently, $\rho(\calX_R) = \{A \in \calY_R \mid A = R:(R:A)\}$, $\eta(\calY_R) = \{I \in \calX_R \mid I = R:(R:I)\}$, and we have a bijective correspondence $\eta(\calY_R) \to \rho(\calX_R), \ I \mapsto I:I$.
\end{lemma}

\begin{proof}
(1) Let $A \in \calY_R$ and choose $a \in W$ so that $a A \subseteq R$. We set $I = aA$. We then have $I^2 = aI$ and $I:I = aA : aA = A : A = A$, whence $I \in \calZ_R$, and $\xi$ is surjective, because $\xi(I)=A$. Let $I,J \in \calZ_R$ and choose $a \in I, b \in J$ so that $I^2 = aI$ and $J^2 = bJ$. Then, $I:I =a^{-1}I$ and $J:J=b^{-1}J$. Hence, if $\xi(I)= \xi(J)$, then $a^{-1}I =b^{-1}J$, so that $I \cong J$ as an $R$-module. Conversely, suppose that $I, J \in \calZ_R$ and $I \cong J$. Then $J = \alpha I$ for some invertible element $\alpha$ of $Q(R)$, whence $\xi(J) = J:J = \alpha I:\alpha I = I:I = \xi(I)$.

(2) (3) Notice that $\eta(\rho(I))= R:(I:I) = R:(R:I)$ for every $I \in \calX_R$ and $$\rho(\eta(A)) = (R:A) :(R:A)= R:A(R:A) = R: (R:A)$$
for every $A \in \calY_R$.

The last assertions follow from the fact that $R:(R:Y)=Y$ for every $R$-submodule $Y$ of $\rmQ(R)$, once $Y = R:X$ for some $R$-submodule $X$ of $\rmQ(R)$ (see Proposition \ref{3.0} (1)).
\end{proof}

\begin{cor}\label{3.10}
The correspondence $\rho$ induces a bijection 
$$\calG_R \to \{A \in \calY_R \mid aA = R:A~\text{for~some}~a \in W\}, \ \ I \mapsto I:I.$$
\end{cor}

\begin{proof}
Let $I \in \calG_R$. We then have, by Proposition \ref{3.3}, $I^2=aI$ and $I = (a):_RI$ for some $a \in I$. Since $I=(a): I= R:a^{-1}I$, $I = R:(R:I)$ by Proposition \ref{3.0} (1). Therefore, setting $A = I:I~(=a^{-1}I)$, because $A = R:I$ by Corollary \ref{2.5}, we get $R:A = R : (R:I) = I = aA$. Hence, $\rho(I)=A$ belongs to the set of the right hand side. By Lemma \ref{3.6}, the induced correspondence is automatically injective, since $I = R:(R:I)$ for every $I \in \calG_R=\calX_R \cap \calZ_R$.

To see the induced correspondence is surjective, let $A \in \calY_R$ and assume that $aA = R:A$ for some $a \in W$. Let $I = aA$; hence $I= \eta(A) \in \calX_R$. We then have $I^2=aI$ and $I:I = aA:aA = A$, so that $I \in \calX_R \cap \calZ_R$ and $\rho(I) = A$. 
\end{proof}

If $R$ is a Gorenstein ring of dimension one, $L = R:(R:L)$ for every  finitely generated $R$-submodule  $L$ of $\rmQ(R)$ such that $\rmQ(R){\cdot} L= \rmQ(R)$. Therefore, by Lemma \ref{3.6}  we readily get the following.

\begin{cor}\label{3.8a}
Suppose that $R$ is a Gorenstein ring of dimension one. Then, $\eta\circ \rho=1_{\calX_R}$ and $\rho\circ \eta = 1_{\calY_R}$, so that the correspondence $\rho : \calX_R \to \calY_R$ is  bijective.
\end{cor}

We note the following.

\begin{prop}\label{3.9}
The following conditions are equivalent.
\begin{enumerate}[{\rm (1)}]
\item $\rho : \calX_R \to \calY_R$ is surjective.
\item $\eta : \calY_R \to \calX_R$ is injective.
\item $A=R:(R:A)$ for every $A \in \calY_R$.
\end{enumerate}
\end{prop}

\begin{proof}
(1) $\Leftrightarrow$ (3) See Lemma \ref{3.6}.

(2) $\Rightarrow$ (3) Let $A \in \calY_R$ and set $L = R:(R:A)$. Therefore, $L = \rho(\eta(A)) \in \calY_R$, while $\eta (A) = R:A= R:L = \eta(L)$, where the second equality follows from Proposition \ref{3.0} (1). Hence, $A = L$, because $\eta$ is injective.

(3) $\Rightarrow$ (2)
 We have $\rho \circ \eta=1_{\calY_R}$ by Lemma \ref{3.6}, so that $\rho$ is surjective.
\end{proof}


We explore one example, which shows that when $\dim R \ge 2$, in general we cannot expect the bijectivity of the correspondence $\rho$.

\begin{ex}\label{3.7}
Let $S=k[X,Y]$ be the polynomial ring over a field $k$. We set $R= k[X^4, X^3Y, XY^3, Y^4]$ and $T = k[X^4,X^3Y, X^2Y^2, XY^3, Y^4]$ in $S$. Let $\m = (X^4,X^3Y,XY^3,Y^4)R$. Then $T = \overline{R}$  and $\m = R:T$. We have $\dim R=2$ and $\depth R_\m=1$, whence $R_\m$ is not Cohen-Macaulay. With this setting the following assertions hold true.

\begin{enumerate}[{\rm (1)}]
\item $\calX_R=\left\{ I \mid I~\text{is~an~ideal~of}~R~\text{with}~\operatorname{ht}_RI \ge 2,~\text{and}~I \not\subseteq \m~\text{or}~IT = I\right\}$. Hence, $\m^\ell \in \calX_R$ for all $\ell >0$.
\item $\calY_R=\left\{T,R\right\}$, and the correspondence $\eta : \calY_R \to \calX_R$ is injective.
\item The isomorphism classes in $\calZ_R$ are $[(X^4,X^6Y^2)]$ and $[R]$, where for each $I \in \calZ_R$, $[I]$ denotes the isomorphism class of $I$ in $\calZ_R$.
\end{enumerate}
\end{ex}


\begin{proof}
$T=\sum_{n \ge 0}S_{4n}$ is the Veronesean subring of $S$ with order $4$, whence $T$ is a normal ring with $\dim T=2$. We get $\m = T_+ \cap R$, where $T_+$ is the maximal ideal $(X^4, X^3Y, X^2Y^2, XY^3, Y^4)T$ of $T$. Because $T = R + kX^2Y^2$ and $\m {\cdot}X^2Y^2 \subseteq \m$, $T = \overline{R}$, the normalization of $R$, and $\m T = \m$. Hence, $R:T= \m$, and $\dim R = \dim R_\m = 2$. However, because $T/R \cong R/\m$, $\depth R_\m = 1$, whence $R_\m$ is not  Cohen-Macaulay. We get $\calY_R =\{T,R\}$, since $\ell_R(T/R)=1$. Therefore, since $\m = R:T$, the correspondence $\eta : \calY_R \to \calX_R$ is injective, and 
 by Lemma \ref{3.6} (1) the isomorphism classes in $\calZ_R$ are exactly $[(X^4, X^6Y^2)]$ (notice that $(X^4, X^6Y^2)=X^4T$) and $[R]$.

Let us check Assertion (1). Firstly, let $I$ be an ideal of $R$ with $\height_RI \ge 2$ such that $I \not\subseteq \m$ or $IT = I$. We will show that $I \in \calX_R$. We may assume $I \ne R$. Suppose that $I \not\subseteq \m$ and let $\fkp \in \Spec R$ such that $I \subseteq \fkp$. Then, $R_\fkp = T_\fkp$, since $\fkp \ne \m$, so that $R_\fkp$ is a Cohen-Macaulay ring with $\dim R_\fkp = 2$. Therefore, $\grade_RI = 2$, and hence  $I \in \calX_R$  by Proposition \ref{2.3}.

Suppose that $IT=I$ and let $f \in \Hom_R(I,R)$. Let $\iota : R \to T$ denote the embedding. Then, the composite map $g :I \overset{f}{\to} R \overset{\iota}{\to} T$ is $T$-linear, because it is the restriction of the homothety of some element of $\rmQ(R)=\rmQ(T)$, while $\grade_TI = \height_TI = 2$, since $\height_RI = 2$. Consequently, because $T:I = T$, we have $g(I) \subseteq I$, so that $f(I) \subseteq I$. Thus, $I \in \calX_R$ by Proposition \ref{2.3}.

Conversely, let $I \in \calX_R$. Therefore, $I$ is a non-zero ideal of $R$ with $R:I = I:I$. Hence, $R:I = R$ or $R:I = T$, because $\calY_R=\{R,T\}$. If $R:I = R$, then $\grade_RI \ge 2$. Therefore, $\height_RI \ge 2$, and $I \not\subseteq \m$, because $\depth R_\m=1$. Suppose that $R:I=T$. Then, $I$ is an ideal of $T$. We have to show $\height_RI \ge 2$. Assume the contrary and choose $\fkp \in \Spec R$ so that $I \subseteq \fkp$ and $\height_R\fkp = 1$. We then have $R_\fkp = T_\fkp$, and  $$R_\fkp : IR_\fkp = [R:I]_\fkp = [I:I]_\fkp =T_\fkp = R_\fkp.$$
This is impossible, because $IR_\fkp$ is a proper ideal in the DVR $R_\fkp = T_\fkp$. Therefore, $\height_RI \ge 2$, which completes the proof of Assertion (1).
\end{proof}

\section{The case where $R$ is a Gorenstein ring of dimension one}

We now concentrate our attention on the case where $R$ is a Gorenstein ring of dimension one.

\begin{prop}\label{3.8} Assume that $R$ is a Gorenstein ring of dimension one. We then have the following.
\begin{enumerate}[{\rm (1)}]
\item $I:I$ is a Gorenstein ring for every $I \in \calG_R$.
\item Let $A \in \calY_R$ and suppose that $A$ is a Gorenstein ring. Then, $A = I:I$ for some $I \in \calG_R$, if $R$ is semi-local. 
\end{enumerate}
Consequently, when $R$ is semi-local, the correspondence $\rho$ induces the bijection $$\calG_R \to \{A \in \calY_R \mid A ~ \text{is a Gorenstein ring} \}.$$
\end{prop}

\begin{proof} 
(1) Let $A = I:I$. Then, by Corollary \ref{3.10}, $R:A = aA$ for some $a \in W$, so that $A$ is a Gorenstein ring (see \cite[Satz 5.12]{HK}; remember that $\Hom_R(A,R) \cong R:A$.)

(2)  We have $R:A = aA$ for some $a \in W$, because $R:A$ is a canonical ideal of  $A$ and $A$ is semi-local. Hence, by Corollary \ref{3.10}, $A = I:I $ for some $I \in \calG_R$.
\end{proof}


When $(R,\m)$ is a Gorenstein local ring of dimension one, we furthermore have the following, which characterizes Gorenstein local rings of dimension one, in which every regular trace ideal is a good ideal. The problem of when $A$ is a Gorenstein ring for every $A \in \calY_R$ is originated in the paper of H. Bass \cite{B}, where one can find many deep observations related to the problem. The equivalence of Assertions (1) and (3) in the following theorem is essentially due to \cite[(7.7) Theorem]{B}.


\begin{thm}\label{3.9a}
Let $R$ be a semi-local Gorenstein ring of dimension one. Then the following conditions are equivalent.
\begin{enumerate}[{\rm (1)}]
\item Every $A \in \calY_R$ is a Gorenstein ring.
\item $\calX_R= \calG_R$.
\end{enumerate}
When $(R,\m)$ is a local ring, one can add the following.
\begin{enumerate}[{\rm (1)}]
\item[{\rm (3)}] $\rme(R) \le 2$.
\end{enumerate}
\end{thm}

\begin{proof}

(2) $\Rightarrow$ (1)
We have by Lemma \ref{3.6} $A = I:I$ for some $I \in \calX_R$, so that by Proposition \ref{3.8} (1) $A$ is a Gorenstein ring.

(1) $\Rightarrow$ (2) Every good ideal of $R$ belongs to $\calX_R$ by Proposition \ref{3.3}. Conversely, let $I \in \calX_R$ and set $A= I:I$.  Then, by Proposition  \ref{3.8} (2), $A = J:J$ for some $J \in \calG_R$, since $A$ is a Gorenstein ring. Therefore, $I=J$, because $I, J \in \calX_R$ and the correspondence $\rho$ is bijective (Corollary \ref{3.8a}).

Suppose that $(R,\m)$ is a local ring.

(3) $\Rightarrow$ (1) See \cite[Lemma 12.2]{GTT1}.

(2) $\Rightarrow$ (3) We may assume that $R : \m = \m : \m$; otherwise, $R$ is a DVR, since $x \m = R$ for some $x \in R:\m$. Therefore, $\m \in \calX_R =\calG_R$, whence $\m^2 = a \m$ for some $a \in \m$. Thus, $\rme(R) =2$, because $R$ is a Gorenstein ring.
\end{proof}

We close this section with a few examples. To state Example \ref{3.3ex}, we need the notion of idealization, which we now briefly explain. Let $R$ be a commutative ring and $M$ an $R$-module. We set $A = R \oplus M$ as an additive group, and define the multiplication in $A$ by $(a,x) {\cdot} (b, y) = (ab, ay + bx) $ for $(a,x), (b, y) \in A$. Then, $A$ forms a commutative ring, which is denoted by $A=R\ltimes M$, and called the idealization of $M$ over $R$.

\begin{Example}\label{3.3ex}
Let $V$ be a DVR with $t$ a regular parameter. Let $R = V \ltimes V$ denote the idealization of $V$ over itself. Then, $R$ is a Gorenstein local ring with $\dim R= 1$, $\rme(R) = 2$, and  $\calX_R = \{t^nV \times V \mid n \ge 0\}$.
\end{Example}

\begin{proof} Because $R \cong V[X]/(X^2)$ where $X$ denotes an indeterminate, $R$ is a Gorenstein local ring with $\dim R= 1$, $\rme(R) = 2$. Let $K = \rmQ(V)$. Then, $\rmQ(R) = K \ltimes K$, and $\overline{R} = V \ltimes K$. Consequently $$\calY_R = \{V \ltimes L \mid L~\text{is~a~finitely~generated}~V\text{-submodule~of}~K~\text{such~that}~V \subseteq L\}.$$ Therefore, $\calX_R = \{t^nV \times V \mid n \ge 0\}$ by Corollary \ref{3.8a}, because $R$ is a Gorenstein local ring with $\dim R = 1$ and $R : [V \ltimes L]=\operatorname{Ann}_V(L/V) \times V$ for every finitely generated $V$-submodule $L$ of $K$ such that $V \subseteq L$.
\end{proof}

\begin{Example}\label{2.11}
Let $k$ be a field.
\begin{enumerate}[{\rm (1)}]
\item Let $R=k[[t^4, t^5, t^6]]$. Then $R$ is a Gorenstein ring, possessing
\[
\calX_{R}=\left\{(t^{8}, t^{9}, t^{10}, t^{11}), (t^{6}, t^{8}, t^{9}), (t^{5}, t^{6}, t^{8}), (t^{4}, t^{5}, t^{6}), R \right\} \cup \left\{ (t^{4}-at^5, t^{6}) \mid a \in k\right\} \ \ \text{and}
\]
\[
\calY_{R}=\left\{k[[t]], k[[t^{2}, t^{3}]], k[[t^{3}, t^{4}, t^{5}]], k[[t^{4}, t^{5}, t^{6}, t^{7}]], R\right\}\cup \left\{ k[[t^{2}+at^3, t^{5}]] \mid a \in k \right\},
\]
and the correspondence $\rho: \calX_R \to \calY_R$ is bijective. 

\item Let $R=k[[t^3, t^4, t^5]]$. Then $R$ is not a Gorenstein ring, possessing 
\[
\calX_{R}=\left\{(t^3, t^4, t^5), R\right\} \ \  \text{and}\ \ 
\calY_{R}=\left\{k[[t]], k[[t^{2}, t^{3}]], R\right\},
\]
and the correspondence $\rho: \calX_R \to \calY_R$ is not  surjective. 
\end{enumerate}

\end{Example}


\begin{proof}
(1) We set $V =k[[t]]$ (the formal power series ring) and $S = k[[t^4,t^5,t^6,t^7]]$. We will show the set $\calY_R$ consists of the rings in the list. Let $\m$ and $\m_S$ denote the maximal ideals of $R$ and $S$, respectively. We begin with the following.

\begin{claim}
The following assertions hold true.
\begin{enumerate}[{\rm (1)}] 
\item[\rm{$(1)$}] Set $B_a = k[[t^2+at^3, t^5]]$ for each $a \in k$. Then $S \subsetneq B_a \subseteq k[[t^2, t^3]]$, $B_a = S+S{\cdot}(t^2+at^3)$, and $\ell_S(B_a/S)= 1$.
\item[\rm{$(2)$}] Let $a,b \in k$. Then $B_a = B_b$ if and only if $a=b$.
\end{enumerate}
\end{claim}

\begin{proof}

(1) We set  $T = k[[(t^2+at^3)^2, t^5, (t^2+at^3)^3, (t^2+at^3){\cdot}t^5]]$. Then, $T \subseteq B_a$, and $T \subseteq S$, since $S = k + t^4V$. Because $$\m_T S + \m_S^2 \supseteq (t^4, t^5, t^6, t^7)S =\m_S = t^4V,$$ we get $\m_T S = \m_S$, whence  $T = S$ (remember that $T/\m_T = S/\m_S =k$). 
Consequently, $T=S \subsetneq k[[t^2,t^3]]$, and $B_a =S+S{\cdot}(t^2+at^3)$, because $t^5 \in \m_S$. Therefore, $\mu_S(B_a) = 2$, and $\ell_S(B_a/S) = 1$, since $\m_S B_a = \m_S \subseteq S$.

(2) Suppose $B_a =B_b$. Then, since the $k$-space $B_a/\m_SB_a$ (resp. $B_b/\m_SB_b$) is spanned by the images of $1$ and $t^2 + at^3$ (resp. $1$ and $t^2 + bt^3$), we have $$t^2 + at^3 = \alpha + \beta(t^2 + bt^3) + \gamma$$ for some $\alpha, \beta \in k$ and $\gamma \in t^4V$. Hence, $\alpha = 0$, $\beta = 1$, and $a = b \beta$, so that $a = b$.
\end{proof}

By this claim, we see $R,S, k[[t^3,t^4,t^5]], B_a ~(a \in k), k[[t^2,t^3]], V \in \calY_R$. The relation of embedding among these rings is the following.



\vspace{-2em}

\[
\xymatrix{
           &
V          & \\
& k[[t^2, t^3]] \ar@{-}[u]   &  \\
B_a = k[[t^2+at^3, t^5]] \ar@{-}[ur] &  & k[[t^3, t^4, t^5]] \ar@{-}[ul]\\
& S=k[[t^4, t^5, t^6, t^7]] \ar@{-}[ur] \ar@{-}[ul] \\
& R  \ar@{-}[u]
}
\]


\noindent
We have to show that $\calY_R$ consists of these rings. To see it, let $A \in \calY_R$ and assume that $R \subsetneq A \subsetneq V$. Then, because $R$ is a Gorenstein local ring with $R:\m = R + kt^7$ and $R \subsetneq A$, we get $S = R + kt^7 \subseteq A$. Let us assume that $S \subsetneq A$ and set $\ell = \ell_S(A/S)$. Then $\ell = 1,2$, since $\ell_S(V/S) = 3$. We write $\m_AV = t^nV$ with an integer $n >0$. We then have $n \le 4$, since $t^4 \in \m_A$. Because $A = k + \m_A \not\subseteq S =k + t^4V$ and $A \ne V$, we furthermore have $n = 2$ or $3$. 

Suppose that $\ell = 1$. If $n = 3$, then choosing an element $f = t^3 + g$ with $g \in t^4V=\m_S$, we see $t^3 \in A$, so that $k[[t^3, t^4, t^5]] \subseteq A$. Therefore, $k[[t^3,t^4,t^5]] = A$, because $\ell_S(A/S) = 1$ and $S \subsetneq k[[t^3,t^4,t^5]] \subseteq A$. Let $n = 2$ and choose an element $f = t^2 + at^3 \in A$ with $a \in k$. Then, $B_a \subseteq A$, and $\ell_S(B_a/S) = 1$ by Claim (1), whence $A = B_a$. Suppose now that $\ell = 2$. Then $\ell_A(V/A) = 1$, whence $\m_A = A : V= t^nV$, so that $$A=k+t^nV=k[[t^n, t^{n+1}, \ldots, t^{2n-1}]]$$ with $n= 2$ or $3$.  This proves that $\calY_R = \{R, S, k[[t^3,t^4,t^5]], B_a ~(a \in k), k[[t^2, t^3]], V \}$.

Because $\calX_R = \{ R:A \mid A \in \calY_R\}$ by Corollary \ref{3.8a} it is direct to show that $\calX_{R}$ consists of the following ideals
$R:V=(t^{8}, t^{9}, t^{10}, t^{11}), R:k[[t^2,t^3]] =(t^{6}, t^{8}, t^{9}),  R:k[[t^3,t^4,t^5]]=(t^{5}, t^{6}, t^{8}), R:S=(t^{4}, t^{5}, t^{6})=\m, R$, and $R:B_a= (t^{4}-at^5, t^{6})$ with $a \in k$. Let us note a proof for the fact that $R : B_a = (t^4-at^5,t^6)$.  We set $I = R : B_a$. Firstly, notice that $B_a = R + R{\cdot}(t^2 + at^3)$, since $t^5, (t^2 + at^3)^2 \in \m$.  We then have $t^6 \in I$, since $R:V = t^8V \subseteq I$. Let $\varphi \in I$ and write $\varphi = \alpha t^4 + \beta t^5 + \gamma t^6 + \delta$ with $\alpha, \beta, \gamma \in k$ and $\delta \in R:V$. Then, $\alpha t^4 + \beta t^5 \in I$, and 
$(\alpha t^4 + \beta t^5)(t^2 + at^3) \in R$ if and only if $(\alpha a + \beta)t^7 \in R$ if and only if $\beta = -\alpha a$, which shows $I = (t^4 -a t^5, t^6)$.

(2) The fact $\calY_{R} =\{k[[t^3, t^4, t^5], k[[t^2,t^3]], V\}$ readily follows from Assertion (1). The assertion on $\calX_{R}$ is a special case of the following.
\end{proof}

\begin{prop}\label{3a}
Let $(R,\m)$ be a one-dimensional Cohen-Macaulay local ring and let $V=\overline{R}$ denote the integral closure  of $R$ in $\rmQ(R)$. Assume that $R \ne V$ but $\m V \subseteq R$. Then $\calX_R=\{\m, R\}$.
\end{prop}

\begin{proof}
Because $R \ne V$, we have $R:\m = \m : \m$, whence $\m \in \calX_R$, so that $\{\m, R \} \subseteq \calX_R$. Let $I \in \calX_R$ and set $A = I:I ~(=R:I)$. If $R = A$, then $\grade_RI \ge 2$, and $I=R$. Suppose that $R \subsetneq A$. Then, $I \subseteq \m$, whence $V \subseteq R :\m \subseteq A= R:I = I:I \subseteq V$. Therefore, $A = V$. Consequently, $I$ is an ideal of $V$, whence $I \cong V \cong \m$ as $V$-modules (remember that $V$ is a direct product of finitely many principal ideal domains). Therefore, $\tau_R(I)= \tau_R(\m)=\m$, because $I \cong \m$ as an $R$-module. Hence $\calX_R = \{\m, R\}$.
\end{proof}

We will use Proposition \ref{3a} later in Section 5, in order to prove Proposition \ref{2.7.5}.




\section{Modules in which every submodule is a trace module}\label{sec2}
In this section, we are interested in the question of, for a given $R$-module $X$,  when every $R$-submodule of $X$ is a trace module in it. As is shown in \cite{L2}, this is the case when $X=R$ and $R$ is a self-injective ring. Our goal is the following, which is known by \cite[Theorem 3.5]{LP} in the case where $R$ is a Noetherian local ring and $X=R$.

\begin{thm}\label{1.3}
Suppose that $R$ is a Noetherian ring and let $X$ be an $R$-module.
Then the following conditions are equivalent.
\begin{enumerate}[{\rm (1)}]
\item Every $R$-submodule of $X$ is a trace module in $X$.
\item Every cyclic $R$-submodule of $X$ is a trace module in $X$. 
\item There is an embedding $$0 \to X \to \bigoplus_{\m \in \Max R} \rmE_R (R/\m) $$
of $R$-modules, where for each $\m \in \Max R$, $\rmE_R(R/\m)$ denotes the injective envelope of the cyclic $R$-module $R/\m$.
\end{enumerate}
\end{thm}

To prove Theorem \ref{1.3}, we need some preliminaries. The following is a direct consequence of Proposition  \ref{2.3}.

\begin{prop}\label{2.3.2} The following assertions hold true.
\begin{enumerate}[{\rm (1)}]
\item Let $Y$ be an $R$-submodule of $X$. If every cyclic $R$-submodule of $Y$ is a trace module in $X$, then $Y$ is a trace module in $X$.
\item Let $Z$ and $Y$ be $R$-submodules of $X$ and assume that $Z \subseteq Y$. If $Z$ is a trace module in $X$, then $Z$ is a trace module in $Y$.
\item $($\cite{L2}$)$ If $R$ is a self-injective ring, then every ideal of $R$ is a trace ideal  in $R$. 
\end{enumerate}
\end{prop}

We begin with the following.

\begin{lem}\label{2.3.5}
Let $Y$ be an $R$-submodule of $X$ and assume that $Y$ is a finitely presented $R$-module. Then the following conditions are equivalent.
\begin{enumerate}[{\rm (1)}]
\item $Y$ is a trace module in $X$.
\item $Y_\fkm$ is a trace module in $X_\fkm$ for all $\fkm \in \Max R$.
\item $Y_\fkp$ is a trace module in $X_\fkp$ for all $\fkp \in \Spec R$.
\end{enumerate}
\end{lem}

\begin{proof}
Let $\iota : Y \to X$ denote the embedding and let $$\iota_* : \Hom_R(Y,Y) \to \Hom_R(Y,X)$$ be the induced homomorphism. We set $C = \Coker ~\iota_*$. By Proposition \ref{2.3}, $Y$ is a trace module in $X$, if and only if $C = (0)$, that is $C_\fkp= (0)$ for all $\fkp \in \Spec R$. On the other hand, since $Y$ is finitely presented, we have $$S^{-1}\left[\Hom_R(Y,Z)\right]= \Hom_{S^{-1}R}(S^{-1}Y,S^{-1}Z)$$ for every $R$-module $Z$ and for every multiplicatively closed subset $S$ in $R$. Hence, the condition that $Y$ is a trace module in $X$ is a local condition. 
\end{proof}

We firstly consider the case where $R$ is a local ring.

\begin{thm}\label{lem2.11}
Let $(R,\m)$ be a Noetherian local ring and $X$  an $R$-module. Then the following conditions are equivalent.
\begin{enumerate}[{\rm (1)}]
\item Every $R$-submodule of $X$ is a trace module in $X$.
\item There is an embedding $0 \to X \to E$ of $R$-modules, where $E=\rmE_R(R/\fkm)$ denotes the injective envelope of $R/\fkm$.
\end{enumerate}
\end{thm}

\begin{proof} 
(1) $\Rightarrow$ (2) We may assume that $X \ne (0)$. It suffices to show that $\ell_R(M) < \infty$ and $\ell_R \left((0):_M \fkm\right)=1$ for every non-zero finitely generated $R$-submodule $M$ of $X$. First of all, we show $\depth_RM=0$. In fact, suppose that $\depth_R M>0$, and let $a \in \fkm$ be a non-zerodivisor on $M$. We then have by Proposition  \ref{2.3} $aM = \tau_X(aM)$ and $M=\tau_X(M)$, since both $aM$ and $M$ are trace modules in $X$, while $\tau_X(aM)=\tau_X(M)$, because $aM \cong M$. Hence, $aM=M$, which is impossible because $M \ne (0)$. We now fix one socle  element $0 \ne x \in (0):_M \fkm$ of $M$. Let $N$ be an arbitrary non-zero $R$-submodule of $M$. Then, since $R/\m$ is a homomorphic image of $N/\m N$ and since $R/\m \cong Rx$, we get a homomorphism $f : N \to M$ such that $f(N) = Rx$, which implies $x \in N$, because $N$ is a trace module in $X$ (see Proposition \ref{2.3}). Therefore, if $\dim_RM >0$, then $x \in \m^n M$ for all $n > 0$, because $\m^n M \ne (0)$, so that $x \in \bigcap_{n >0}\m^n M =(0)$, which is a contradiction. Hence, $\dim_RM=0$, that is $\ell_R(M) < \infty$. The above observation also shows that $x \in Ry$ for every $0 \ne y \in (0):_M\m$, whence $\ell_R\left((0):_M\m\right) = 1$, and therefore, $X$ is an $R$-submodule of $E$.

(2) $\Rightarrow$ (1) By Proposition \ref{2.3.2} (2), we may assume $X = E$. Let $Y$ be an $R$-submodule of $E$. It suffices to show that $f(Y) \subseteq Y$ for all $f \in \Hom_R(Y,E)$. We take a homomorphism $g : E \to E$ so that $f = g \circ  \iota$, where $\iota : Y \to E$ denotes the embedding. Let $\widehat{R}$ denotes the $\m$-adic completion of $R$, and remember that $E$ is an $\widehat{R}$-module such that $$\Hom_R(E,E) = \Hom_{\widehat{R}}(E,E) = \widehat{R}.$$
 Choose $\alpha \in \widehat{R}$ so that $g$ is the homothety by $\alpha$. We then have $\alpha Y \subseteq Y$, because every $R$-submodule of $E$ is actually an $\widehat{R}$-submodule of $E$. Therefore $$f(Y) = g(Y) = \alpha Y \subseteq Y,$$ and hence $Y$ is a trace module in $E$. 
\end{proof}

We are now ready to prove Theorem \ref{1.3}.

\begin{proof}[Proof of Theorem $\rm \ref{1.3}$]
(1) $\Leftrightarrow$ (2) See Proposition \ref{2.3.2} (1).

(3) $\Rightarrow$ (1) Let $\m \in \Max R$. We then have the embedding $0 \to X_\m \to \rmE_{R_\m}(R_\m/\m R_\m)$, since $$[\bigoplus_{\n \in \Max R} \rmE_R (R/\n)]_\m = \rmE_{R_\m}(R_\m/\m R_\m).$$  Therefore, by Theorem \ref{lem2.11}, for every cyclic $R$-submodule $Y$ of $X$, $Y_\m$ is a trace module in $X_\m$ for all $\m \in \Max R$, so that Lemma \ref{2.3.5} guarantees that $Y$ is a trace module in $X$. Hence, by Proposition \ref{2.3.2} (2), every $R$-submodule of $X$ is a trace module in $X$.

(1) $\Rightarrow$ (3) Let $\m \in \Max R$. Since every cyclic $R_\m$-submodule of $X_\m$ is a localization of a cyclic $R$-submodule of $X$, by Lemma \ref{2.3.5} every $R_\m$-submodule of $X_\m$ is a trace module in $X_\m$. Therefore, by Theorem \ref{lem2.11}, for every $\m \in \Max R$ we have $$\Ass_{R_\m}X_\m \subseteq \{\m R_\m\}\ \ \text{and}\ \ \ell_{R_\m}\left((0):_{X_\m}\m R_\m\right) \le 1.$$ Consequently, $\Ass_R X \subseteq \Max R$ and $\ell_R\left((0):_X \m \right) \le 1$ for all $\m \in \Max R$, so that $$\rmE_R(X) \cong \bigoplus_{\m \in \Max R}\rmE(R/\m)^{\oplus \mu (\m)},$$  where $\E_R(X)$ denotes the injective envelope of $X$, and $\mu (\m) \in \{0, 1\}$.
\end{proof}

The following is a direct consequence of Theorem \ref{1.3}.

\begin{cor}[{cf. \cite[Theorem 3.5]{LP}}]\label{cor2.12}
For a Noetherian ring $R$, the following conditions are equivalent.
\begin{enumerate}[{\rm (1)}]
\item Every ideal of $R$ is a trace ideal in $R$.
\item $R$ is a self-injective ring.
\end{enumerate}
\end{cor}

For the implication (1) $\Rightarrow$ (2) in Corollary \ref{cor2.12}, we cannot remove the assumption that $R$ is a Noetherian ring. To explain more precisely about this phenomenon, let $R$ be a commutative ring. We say that $R$ is a {\em Von Neumann regular} ring, if for each $a \in R$, there exists an element $b \in R$ such that $a = aba$ (cf. \cite{Neumann}). Here, we need only the definition, but interested readers can find in \cite{Endo} a basic characterization of Von Neumann regular rings.

\begin{lemma}
Let $R$ be a Von Neumann regular ring. Then $\tau_R(I) = I$ for every ideal $I$ of $R$.

\end{lemma}

\begin{proof}
Let $\varphi : I \to R$ be an $R$-linear map and $a \in I$. Then, $a = aba$ for some $b \in R$, so that $\varphi(a) = a\varphi(ba) \in I$. Thus, $\varphi(I) \subseteq I$.
\end{proof}

We have learned the following example from M. Hashimoto.

\begin{ex}\label{hashimoto}
Let $K$ be a commutative ring and assume that there exists an integer $p \ge 2$ such that $a^p = a$ for every $a \in K$. We consider the direct product $S=\prod_{i \in \Lambda}K_i$ of infinitely many copies $\{K_i =K\}_{i \in \Lambda}$ of $K$, and set $R = \Bbb Z{\cdot}1 + \bigoplus_{i \in \Lambda}K_i$ in $S$. Then, $R$ is a subring of $S$, and $R$ is Von Neumann regular, since $a^p=a$ for every $a \in S$.  We have  that $S$ is an essential extension of $R$, but $R \ne S$, because $\Lambda$ is infinite. Therefore, $R$ is not a self-injective ring.
\end{ex}

Let us note one more example. The following fact is known, when ${\rm ch} k=2$ and $\alpha_i = 1$ for every $i \in \Lambda$. Indeed, with the same notation as Example \ref{not self-injective}, if ${\rm ch} k = 2$ and $\alpha_i = 1$ for all $i \in \Lambda$, then $R = k[\{T_i\}_{i \in \Lambda}]/(T_i^2 - 1 \mid i \in \Lambda)$ where $T_i = X_i -1$ for each $i \in \Lambda$, so that $R=k[G]$, the group algebra of the direct sum $G=\bigoplus_{i \in \Lambda} C_i$ of infinitely many copies of the cyclic group $C_i=\Bbb Z/(2)$. Therefore, thanks to \cite[Theorem]{F}, $R$ is not self-injective. We have learned this result from K. Kurano, and we are grateful to him, since the method of proof given in \cite{F} works also in the setting of Example \ref{not self-injective}, as we will briefly confirm below.

\begin{ex}\label{not self-injective} Let $\Lambda =\{1, 2, 3, \ldots \}$ be the set of positive integers. Let $\{X_i\}_{i \in \Lambda}$ be a family of indeterminates and $\{\alpha_i\}_{i \in \Lambda}$  a family of positive integers. We set $S= k[\{X_i\}_{i \in \Lambda}]$ over a field $k$, $\fka = (X_i^{\alpha_i +1} \mid i \in \Lambda)$, and consider the ring $R = S/\fka$. Then, $R$ is not a self-injective ring, but $\tau_R(I) = I$ for every ideal $I$ of $R$.
\end{ex}

\begin{proof}
Let $x_i$ denote, for each $i \in \Lambda$, the image of $X_i$ in $R$. For each $n \in \Lambda$, we set $R_n =k[x_1, x_2, \ldots, x_n]$ in $R$. Then, $R = \bigcup_{n \in \Lambda}R_n$, and $$R_n = k[X_1, X_2, \ldots, X_n]/(X_1^{\alpha_1 + 1}, X_2^{\alpha_2 + 1}, \ldots, X_n^{\alpha_n + 1}),$$ so that $R_n$ is a self-injective ring for every $n \in \Lambda$. Let $a \in R$ and assume that $a \in R_n$. Then $$(0):_R[(0):_R a] \subseteq \bigcup_{\ell \ge n}\left\{(0):_{R_\ell}[(0):_{R_\ell} a]\right\},$$ whence  $(0):_R[(0):_Ra] = (a)$, because  $(0):_{R_\ell}[(0):_{R_\ell} a] =a{\cdot}R_\ell$ for all $\ell \ge n$ (here we use the fact that $R_\ell$ is a self-injective ring). Therefore, $\tau_R(I)= I$ for every ideal $I$ of $R$, because $\tau_R((a)) = (0):_R[(0):_R a]=(a)$ for each $a \in R$.

To see that $R$ is not self-injective,
we set for each $n \in \Lambda$
$$a_n= \begin{cases}
1, &  \ \text{if}\  n=1\\
1+ x_1+ x_1x_2 + x_1x_2x_3+ \ldots + x_1x_2\cdots x_{n-1}, & \ \text{if}\ n>1
\end{cases}
$$
and set $I_n =(x_1^{\alpha_1}, x_2^{\alpha_2}, \ldots, x_n^{\alpha_n})$. Then, $I_n \subseteq I_{n+1}$, and $I = \bigcup_{n \in \Lambda}I_n$, where $I = (x_i^{\alpha_i} \mid i \in \Lambda)$. We then have $a_{n+1}x = a_nx$ for every $x \in I_n$, which one can show by a simple use of induction on $n$, since $x_i^{\alpha_i + 1}= 0$ for all $i \in \Lambda$. Therefore, we may define the $R$-linear map $\varphi : I \to R$ so that $\varphi (x) = a_n x$ if  $x \in I_n$. We now assume that $R$ is a self-injective ring. Then, there must exist an element $a \in R$ such that $ax = \varphi(x)$ for every $x \in I$,  namely $ax = a_n x$ for every $x \in I_n$. Choose $n \in \Lambda$ so that $a \in R_n$. Then, because $(a-a_{n+2})x_{n+2}^{\alpha_{n+2}}= 0$, we get $a- a_{n+2} \in (0):_Rx_{n+2}^{\alpha_{n+2}}= (x_{n+2})$. Let $f \in k[X_1, X_2, \ldots, X_{n}]$ such that $a$ is the image of $f$ in $R$. Then  
$$f = 1+X_1+ X_1X_2+ \ldots +X_1X_2\cdots X_{n+1} + X_{n+2}g +h$$
for some $g \in S$ and $h \in \fka$. Substituting $X_i$ by $0$ for all $i \ge n+2$, we may assume that $g=0$ and $h \in (X_1^{\alpha_1 +1}, X_2^{\alpha_2 +1}, \ldots, X_{n+1}^{\alpha_{n+1} +1})T$, where $T = k[X_1, X_2, \ldots, X_{n+1}]$, that is
$$f = 1+X_1+ X_1X_2+ \ldots +X_1X_2\ldots X_{n+1} + \sum_{i=1}^{n+1}X_i^{\alpha_i + 1}h_i$$
with $h_i \in T$. This is, however, impossible, because $f \in k[X_1, X_2, \ldots, X_n]$ and the monomial $X_1 X_2 \cdots X_{n+1}$ is not involved in the polynomial $\sum_{i=1}^{n+1}X_i^{\alpha_i + 1}{h_i}$. Thus, $R$ is not a self-injective ring.
\end{proof}

It seems interesting, but hard, to ask for a complete characterization of (not necessarily Noetherian) commutative rings, in which every ideal is a trace ideal.



\section{Surjectivity of the correspondence $\rho$ in dimension one}
In this section, let $(R,\m)$ be a Cohen-Macaulay local ring of dimension one. We are interested in the question of when the correspondence $\rho : \calX_R \to \calY_R$ is  bijective. The second example in Example \ref{2.11} seems to suggest that $R$ is a Gorenstein ring, if $\dim R=1$ and $\rho$ is  bijective.  Unfortunately, this is still not the case, as we show in the following. Here, we say that a one-dimensional Cohen-Macaulay local ring $(R,\m)$ has maximal embedding dimension, if $\m^2=a\m$ for some $a \in \m$ (\cite{S}). We refer to \cite{GMP, GTT1} for the notion of almost Gorenstein local ring.

\begin{prop}[{cf. \cite[Example 4.7]{KT}}]\label{2.7.5}
Let $K/k$ be a finite extension of fields. Assume that $k \ne K$ and there is no intermediate field $F$ such that $k \subsetneq F \subsetneq K$. Let $B = K[[t]]$ be the formal power series ring over $K$ and set $R = k[[Kt]]$ in $B$. Set $n = [K:k]$. We then have the following.
\begin{enumerate}[{\rm (1)}]
\item $R$ is a Noetherian local ring with $B = \overline{R}$ and $\fkm=tB$, where $\m$ denotes the maximal ideal of $R$. Hence $B = \m : \m = R : \m$.
\item $R$ is an almost Gorenstein local ring, possessing maximal embedding dimension $n \ge 2$. 
\item $R$ is not a Gorenstein ring, if $n \ge 3$.
\item $\calX_R=\{\m,R \}$ and $\calY_R=\{B,R\}$, so that  $\rho:\calX_R \to \calY_R$ is a bijection.
\end{enumerate}
\end{prop}


\begin{proof} Let $\omega_1=1, \omega_2, \ldots, \omega_n$ be a $k$-basis of $K$. Then $R = k[[\omega_1t, \omega_2t, \ldots, \omega_nt]]$, whence $R$ is a Noetherian complete local ring. Since $B/\m B\cong K$, $B = \sum_{i=1}^nR\omega_i$, so that $tB = \m$. Hence, $\m$ is also an ideal of $B$, $\m = \m B=tB$, and  $\m^2 = t\m$. Because $B$ is a module-finite extension of $R$ and $\omega_i = \frac{\omega_it}{\omega_1t} \in \rmQ(R)$ for all $1 \le i \le n$, we have $B = \overline{R}$. Therefore, $R$ is an almost Gorenstein ring by \cite[Corollary 3.12]{GMP}, possessing maximal embedding dimension $\rme(R) = n$. Consequently, $R$ is not a Gorenstein ring, if $n \ge 3$. We get $\calX_R = \{\m, R\}$ by Proposition \ref{3a}, because $R \ne B$ but $\m B \subseteq R$. The assertion that $\calY_R = \{B,R\}$ is due to \cite[Example 4.7]{KT}. Let us note a brief proof for the sake of completeness. Let $A \in \calY_R$ and let $\n$ denote the maximal ideal of $A$. We then have $\n = \m$, because $\n = \m_B \cap A= \m \cap A = \m$. Consequently, we have an extension $k=R/\m \subseteq A/\m \subseteq K = B/\m$ of fields, so that $R/\m = A/\m$, or $A/\m = B/\m$ by the choice of the extension $K/k$. Hence, $R = A$ or $A = B$, and thus $\calY_R = \{R, B\}$. Therefore, because $\m : \m = tB:tB=B$ and $R:R=R$, the correspondence $\rho:\calX_R \to \calY_R$ is a bijection.
\end{proof}


In what follows, we intensively explore the question of when the correspondence $\rho : \calX_R \to \calY_R$ is  bijective. The goal is the following, which essentially shows that except the case of Proposition \ref{2.7.5}, the surjectivity of $\rho$ implies the Gorenstein property of the ring $R$.

\begin{thm}\label{1.2}
Let $(R, \fkm)$ be a Cohen-Macaulay local ring of dimension one. We set $B=\fkm:\fkm$ and let $\rmJ(B)$ denote the Jacobson radical of $B$. Then the following assertions are equivalent.
\begin{enumerate}[{\rm (1)}]
\item $\rho : \calX_R \to \calY_R$ is  bijective.
\item $\rho : \calX_R \to \calY_R$  is surjective.
\item Either $R$ is a Gorenstein ring, or the following two conditions are satisfied.
\begin{enumerate}[{\rm (i)}]
\item $B$ is a $\operatorname{DVR}$ and $\rmJ(B)=\fkm$.
\item There is no proper intermediate field between $R/\fkm$ and $B/\rmJ(B)$.
\end{enumerate} 
\end{enumerate}
When this is the case, $R$ is an almost Gorenstein local ring.
\end{thm}

We set $B=\fkm:\fkm$. Let $\rmJ(B)$ be the Jacobson radical of $B$. To prove Theorem \ref{1.2}, we need some preliminaries. Let us begin with the following.

\begin{lemma}\label{2.7.4} Suppose that $R$ is not a $\operatorname{DVR}$. Then $R \ne B$ and $\ell_{R} \left(B/R\right) = \rmr(R)$.
\end{lemma}

\begin{proof}
We have $R:\m =\m :\m$, since $R$ is not a DVR. The second assertion is clear, since $\ell_R\left((R:\m)/R\right) = \rmr(R)$. 
\end{proof}

\begin{prop}\label{2.7.6}
Suppose that $R \ne B$ and that $\rho:\calX_R \to \calY_R$ is  surjective. Then there is no proper intermediate ring between $R$ and $B$. 
\end{prop}

\begin{proof}
We have $\fkm, R \in \calX_R$ and $B, R \in \calY_R$. Let $A$ be an extension of $R$ such that  $R\subsetneq A \subseteq B$. We write $A = \rho(I)=R:I$ with $I \in \calX_R$. Then $I \subseteq \m$, since $A \ne R$. Therefore, $A= R:I \supseteq R:\m = B$, so that $A = B$. 
\end{proof}

The following is the heart of our argument.

\begin{thm}\label{2.8}
Let $(R, \fkm)$ be a non-Gorenstein Cohen-Macaulay local ring of dimension one. Assume that $R$ is $\m$-adically complete and there is no proper intermediate ring between $R$ and $B$. Then the following assertions hold true.
\begin{enumerate}[{\rm (1)}]
\item $B=\overline{R}$, and $B$ is a $\operatorname{DVR}$ with $\rmJ(B)=\fkm$.
\item $[B/\m:R/\fkm]=\rmr(R)+1\ge 3$. 
\item There is no proper intermediate field between $R/\fkm$ and $B/\m$.
\item $\calX_R=\{\m, R\}$ and the correspondence $\rho$ is bijective.
\end{enumerate}
\end{thm}

\begin{proof}
We have $\m B = \m$, and $R \ne B$, since $R$ is not a DVR (Lemma \ref {2.7.4}). Let $x \in B \setminus R$. Then $B=R[x]$  and  $B/\fkm=k[\overline{x}]$, where $k = R/\m$ and $\overline{x}$ denotes the image of $x$ in $B/\fkm$. Let $n~(>0)$ be the degree of the minimal polynomial of $\overline{x}$ over $k$. We then have $$B=R+ Rx + Rx^2+ \cdots +Rx^{n-1}$$ and $n = \mu_R(B)$, so that $n-1=\rmr(R)$ by Lemma \ref{2.7.4}. Therefore, $n \ge 3$ since $R$ is not a Gorenstein ring, so that $x^2 \not\in R$ since the elements $1, x, \ldots, x^{n-1}$ form a minimal system of generators of the $R$-module $B$. Hence $$B=R[x^2]=R+ Rx^2 + Rx^4+ \cdots +Rx^{2(n-1)}.$$ Let us write $x=\sum_{i=0}^{n-1} c_i x^{2i}$ with $c_i \in R$. We then have $x(1-ax)=c_0$, where $a=\sum_{i=1}^{n-1} c_i x^{2i-2}$.
We will show that $x \not\in \rmJ(B)$. If $c_0 \not\in \fkm$, then $x$ is a unit of $B$, whence $x \not\in \rmJ(B)$. Assume that $c_0 \in \fkm$. Then, if $x \in \rmJ(B)$, $1-ax$ is a unit of $B$, so that $x=(1-ax)^{-1} c_0 \in \fkm B=\fkm$, which is a contradiction. Therefore, $x \not\in \rmJ(B)$ for all $x \in B \setminus R$, which shows $\rmJ(B) \subseteq R$, whence $\rmJ(B) = \fkm$. Therefore, we have $B=\fkm:\fkm=\rmJ(B):\rmJ(B)$. Hence, $B_M=MB_M:MB_M$ for all $M \in \Max B$, which implies  the local ring $B_M$ is a DVR (see Lemma \ref{2.7.4}). Therefore, because $B$ is integrally closed in $\rmQ(B) = \rmQ(R)$, we get $B=\ol{B}=\ol{R}$.

Since $R$ is $\m$-adically complete, we have a decomposition 
$$
B = B_1\times B_2 \times \cdots \times B_{\ell}
$$
of $B$ into a finite product of DVR's $\{B_j\}_{1 \le j \le \ell}$. We want to show that $\ell = 1$. Let $\mathbf{e}_j=(0, \dots, 0, 1_{B_j}, 0, \dots, 0)$ in $B$ and set $\mathbf{e}=\sum_{j=1}^{\ell} \mathbf{e}_j$.  Assume now that $\ell \ge 2$. We then have $B=R[\mathbf{e}_1]$, since $\mathbf{e}_1 \not\in R$ and since there is no proper intermediate ring between $R$ and $B$. Hence $B=R \mathbf{e} + R\mathbf{e}_1$, since $\mathbf{e}_1^2 =\mathbf{e}_1$. This is however impossible, because $$\mu_R(B) = \ell_{R} (B/\fkm B)=1+\rmr(R)>2.$$ Thus, $\ell=1$,  that is $B=\ol{R}$ is  a DVR with the maximal ideal $\rmJ(B)=\fkm$. It remains the proof of Assertions (3) and (4). Assume that there is contained a field $F$ such that $R/\fkm \subseteq F \subseteq B/\m$.  We consider the natural epimorphism $\varepsilon:B \to B/\m$ of rings. Then, since $\varepsilon^{-1}(F)$ is an intermediate ring between $R$ and $B$, either $\varepsilon^{-1}(F)=R$, or $\varepsilon^{-1}(F)=B$, which shows either $F=R/\fkm$, or $F=B/\m$.

Let $I \in \calX_R$ and assume that $I \ne R$. Then, since $I \subseteq \m$, we have 
$$B = \m :\m = R:\m \subseteq R:I = I:I \subseteq \overline{R}=B,$$ whence $I:I= B$, so that $I$ is an ideal of $B$. Let us write $I = aB$ with $0 \ne a \in B$. We then have 
$$B = R:I = R: aB = a^{-1}(R:B) = a^{-1}\m,$$ since $\m =R:B$, so that $\m = aB=I$. Thus, $\calX_R = \{\m, R\}$, which shows  the correspondence $\rho$ is bijective. This completes the proof of Theorem \ref{2.8}.
\end{proof}

We are now ready to prove Theorem \ref{1.2}.

\begin{proof}[Proof of Theorem $\rm \ref{1.2}$]
(1) $\Rightarrow$ (2) This is clear.

$(3) \Rightarrow (1)$ See  Lemma \ref{3.6}, Proposition \ref{2.7.6}, and Theorem \ref{2.8} (4).

$(2) \Rightarrow (3)$ We may assume that $R$ is not a Gorenstein ring. Passing to the $\m$-adic completion $\widehat{R}$ of $R$, without loss of generality we may also assume that $R$ is $\m$-adically complete. Then by Proposition \ref{2.7.6}, there is no proper intermediate ring between $R$ and $B$, so that the assertion follows from Theorem \ref{2.8}.

If $\rho$ is bijective but $R$ is not a Gorenstein ring, we then have $B = \m : \m$ is a DVR, so that  $R$ is an almost Gorenstein ring by  \cite[Theorem 5.1]{GMP}. 
\end{proof}

We note the following, which is a direct consequence of Theorem \ref{1.2}.

\begin{cor}\label{cor2.9}
Let $(R, \fkm)$ be a Cohen-Macaulay local ring of dimension one. Suppose that one of the following conditions is satisfied. 
\begin{enumerate}[{\rm (i)}]
\item The field $R/\fkm$ is algebraically closed.
\item $\ol{R}$ is a local ring, and $R/\m \cong \ol{R}/\n$, where $\n$ denotes the maximal ideal of $\ol{R}$.
\end{enumerate}
Then the following assertions are equivalent.
\begin{enumerate}[{\rm (1)}]
\item $R$ is a Gorenstein ring.
\item The correspondence $\rho : \calX_R \to \calY_R$ is bijective.
\item The correspondence $\rho : \calX_R \to \calY_R$ is surjective.
\end{enumerate}
\end{cor}

When $R$ is a numerical semigroup ring over a field, Condition (ii) of Corollary \ref{cor2.9} is always satisfied.


\section{Anti-stable rings}\label{add}
Let $R$ be a commutative ring and let $\calF_R$ denote the set of regular ideals of $R$. Then, because $(R:I){\cdot}(I:I) \subseteq R:I$, for every $I \in \calF_R$ the $R$-module $R:I$ has also the structure of an $(I:I)$-module. Keeping this fact together with the natural identifications $R:I = \Hom_R(I,R)$ and $I:I = \End_RI$ in our mind, we give the following.

\begin{defn}
We say that $R$ is an {\em anti-stable} (resp. {\em strongly anti-stable}) ring, if  $R:I$ is an invertible $I:I$-module (resp. $R:I \cong I:I$ as an $(I:I)$-module) for every $I \in \calF_R$.
\end{defn}

\noindent
Therefore, every Dedekind domain is anti-stable, and every UFD is a strongly anti-stable ring. Notice that when $R$ is a Noetherian semi-local ring, $R$ is anti-stable if and only if it is strongly anti-stable. Indeed, let $I \in\calF_R$, and set $A=I:I$, $M = R:I$. Then, $A$ is also a Noetherian semi-local ring, and therefore, because $M$ has rank one as an $A$-module, $M$ must be cyclic and free, once it is an invertible module over $A$.

Let  us  recall here that $R$ is said to be a {\em stable} ring, if every ideal $I$ of $R$ is {\em stable}, that is projective over its endomorphism ring $\End_RI$ (\cite{SV}). When $R$ is a Noetherian semi-local ring and $I \in \calF_R$, $I$ is a stable ideal of $R$ if and only if $I \in \calZ_R$, that is $I^2=aI$ for some $a \in I$ (\cite{L}, \cite[Proposition 2.2]{SV}). Our definition of anti-stable rings is, of course, different from that of stable rings. However, we shall later show in Corollary \ref{final} that the anti-stability of rings implies the stability of rings, under suitable conditions. 

First of all, we will show that $R$ is a strongly anti-stable ring if and only if every $I \in \calF_R$ is isomorphic to a trace ideal in $R$.

\begin{lemma}\label{5.0}
Let $I \in \calF_R$ and set $A = I:I$. Then the following conditions are equivalent.
\begin{enumerate}[{\rm $(1)$}]
\item $I \cong J$ as an $R$-module for some $J \in \calX_R$.
\item $I \cong \tau_R(I)$ as an $R$-module.
\item $R:I \cong A$ as an $R$-module.
\item $R:I \cong A$ as an $A$-module.
\item $R:I = a A$ for some unit $a$ of $\rmQ(R)$.
\end{enumerate}
\end{lemma}

\begin{proof}
(1) $\Leftrightarrow$ (2) Since $\tau_R(I) \in \calX_R$, the implication (2) $\Rightarrow$ (1) is clear. Since $J = \tau_R(J)$ for every $J \in \calX_R$ (Proposition \ref{2.3}), we have $\tau_R(I) = J$, if $J \in \calX_R$ and $I \cong J$ as an $R$-module, whence the implication (1) $\Rightarrow$ (2) follows.

(4) $\Rightarrow$ (3) This is clear.

(3) $\Rightarrow$ (4) Because the given isomorphsim $R : I \to A$ of $R$-modules is the restriction of the homothety of some unit $a$ of $\rmQ(R)$, it must be also a homomorphism of $A$-modules, whence $R:I \cong A$ as an $A$-module.

(4) $\Leftrightarrow$ (5) This is now clear.

(1) $\Rightarrow$ (3) We have $I = aJ$ for some unit $a$ of $\rmQ(R)$, whence $R:I = R : aJ = a^{-1}(R:J)$, and $I:I= aJ:aJ= J:J$. Thus, $R:I \cong I:I$ as an $R$-module, because $R:J = J:J$.

(5) $\Rightarrow$ (2) We have $\tau_R(I) = (R:I)I = aA{\cdot}I= aI$, whence $\tau_R(I)\cong I$ as an $R$-module.   
\end{proof}

For a Noetherian ring $R$, we set $\operatorname{Ht_1}(R) =\{\fkp \in \Spec R \mid \height_R\fkp = 1\}$. Let us note the following example of strongly anti-stable rings. We include a brief proof.

\begin{ex}[{\cite[Corollary 3.10]{KT2}}]\label{5.12}
For a Noetherian normal domain $R$, $R$ is a strongly anti-stable ring if and only if $R$ is a ${\rm UFD}$.
\end{ex}

\begin{proof}
Suppose that $R$ is a strongly anti-stable ring and let $\fkp \in \operatorname{Ht_1}(R)$. Then, since $R$ is normal, the $R$-module $\fkp$ is reflexive with  $\fkp: \fkp = R$, while $R: \fkp \cong \fkp: \fkp$ by Lemma \ref{5.0}. Hence, $\fkp \cong R$, so that $R$ is a UFD. Conversely, suppose that $R$ is a UFD and let $I \in \calX_R$. Then, $I  \cong J$ for some ideal $J$ of $R$ with $\grade_RJ\geq 2$, so that $I \cong \tau_R(I)$, since $J \in \calX_R$ by Corollary \ref{2.5}. Thus, $R$ is a strongly anti-stable ring.
\end{proof}

We explore one example of anti-stable rings which are not strongly anti-stable.

\begin{Example}\label{5.0a}
Let $k$ be a field and $S = k[t]$ the polynomial ring. Let $\ell \ge 2$ be an integer and set $R = k[t^2, t^{2\ell +1}]$. We consider the maximal ideal $I = (t^2-1, t^{2\ell + 1}-1)$ in $R$. Then, $\tau_R(I) = R$, and $I \not\cong J$ as an $R$-module for any $J \in \calX_R$. Therefore, $R$ is not a strongly anti-stable ring, while $R$ is an anti-stable ring, because $\dim R = 1$ and for every $M \in \Max R$, $R_M$ is an anti-stable local ring. See Theorem \ref{5.11} for details.
\end{Example}

\begin{proof}
Let $\fkp \in \Spec R$. If $I \not\subseteq \fkp$, then $IR_\fkp = R_\fkp$. If $I \subseteq \fkp$, then $t^2 \not\in \fkp =I$, whence $R_\fkp = S_\fkp$ is a DVR, because $R:S =(t^2, t^{2\ell+1})R$, so that $IR_\fkp \cong R_\fkp$. We now notice  that $I \subseteq \tau_R(I) \subseteq R$. Hence, either $I = \tau_R(I)$ or $\tau_R(I) = R$. If $I = \tau_R(I)$, then setting $\fkp =I$, we get $R_\fkp$ is a DVR and $IR_\fkp = \tau_{R_\fkp}({IR_\fkp}) \subsetneq R_\fkp$, while $\tau_{R_\fkp}(IR_\fkp)= R_\fkp$, because $IR_\fkp \cong R_\fkp$. This is absurd. Hence $\tau_R(I) = R$. Consequently, $I \not\cong J$ for any $J \in \calX_R$. In fact, if $I \cong J$ for some $J \in \calX_R$, then $J = \tau_R(J)= \tau_R(I)=R$, so that $\mu_R(I) = 1$. We write $I = fR$ with some monic polynomial $f \in R$. Let $\overline{k}$ denote the algebraic closure of $k$ and choose $a \in \overline{k}$ so that $f(a)=0$. Then, since $a^2 = a^{2\ell + 1}=1$, we get $a=1$, whence $f = (t-1)^n$ with $0 < n  \in H$, where $H = \left<2, 2\ell + 1\right>$ denotes the numerical semigroup generated by $2, 2\ell+1$. Therefore,  $2-n, (2\ell + 1)-n \in H$, because  $t^2-1, t^{2\ell + 1} -1 \in fR$. Hence, $n=2$, and $2\ell + 1 \in 2+H$, which is impossible.  Thus, $I$ is not a principal ideal of $R$, and $I \not\cong J$ for any $J \in \calX_R$.
\end{proof}

The key in our argument is the following, which plays a key role also in \cite{GI}.

\begin{lemma}\label{5.1}
Let $R$ be a strongly anti-stable ring. Then the correspondence $\rho : \calX_R \to \calY_R$ is surjective. More precisely, let $A \in \calY_R$ and set $J = R:A$. Then $J \in \calG_R=\calX_R \cap \calZ_R$.
\end{lemma}

\begin{proof}
Let $A \in \calY_R$ and choose $b \in W$ so that $bA \subseteq R$. Then, since $bA \in \calF_R$, by Lemma \ref{5.0} $bA \cong J$ as an $R$-module for some $J \in \calX_R$. Let us write $J = aA$ with $a$ a unit of $\rmQ(R)$ (hence $a \in J \cap W$). We then have $J:J=aA:aA = A:A= A$, whence $A=J:J= R:J=R:aA= a^{-1}(R:A)$, so that $R:A=aA=J \in \calX_R \cap \calZ_R$. Therefore, $\rho(J) = J:J = A$, and the correspondence $\rho : \calX_R \to \calY_R$ is surjective. 
\end{proof}

Let us recall one of the fundamental results on stable rings, which we need to prove Theorem \ref{5.6}.

\begin{prop}[{\cite[Lemma 3.2, Theorem 3.4]{SV}}]\label{5.2a}
Let $R$ be a Cohen-Macaulay semi-local ring and assume that $\dim R_M = 1$ for every $M \in \Max R$. If $\rme(R_M) \le 2$ for every $M \in \Max R$, then $R$ is a stable ring.
\end{prop}

We should compare the following theorem with [11, Theorem 3.6].

\begin{thm}\label{5.6}
Let $R$ be a Cohen-Macaulay local ring of dimension one.  Then, $R$ is an anti-stable ring, if and only if $\rme(R) \le 2$.
\end{thm}

\begin{proof} Suppose that $\rme(R) \le 2$. Let $I \in \calF_R$ and set $A = I:I$. Then, by Proposition \ref{5.2a}  $R$ is a stable ring. Hence, $I^2 = aI$ for some $a \in I$, whence $A = a^{-1}I$. Therefore, $I \cong A$ as an $R$-module. We now consider $J = (R:I)I$. Then, $J = \tau_R(I) \in \calX_R$, whence $$J : J = R:J = R : (R:I)I =  [R:(R:I)]:I = I:I,$$ where the last equality follows from the fact that $R$ is a Gorenstein ring. Consequently, $A = J:J \cong J$ (since $J \in \calF_R$), so that $I \cong J = \tau_R(J)$. Thus, $R$ is an anti-stable ring.

Conversely, suppose that $R$ is an anti-stable ring. First of all, we will show that $R$ is a Gorenstein ring. Assume the contrary. Then, passing to the $\m$-adic completion of $R$,  by Proposition \ref{2.7.6} and Theorem \ref{2.8} we get $\calX_R = \{\m, R\}$. Consequently, either $I \cong \m$ or $I \cong R$, for every ideal $I \in \calF_R$. We set $n = \mu_R(\m)$ and write $\m =(x_1, x_2,\ldots, x_n)$ with non-zerodivisors $x_i$ of $R$. Then, $n >2$ since $R$ is not a Gorenstein ring, and setting $I = (x_1, x_2,\ldots, x_{n-1})$, we have either $I \cong \m$ or $I \cong R$, both of which violates the fact that $n = \mu_R(\m) >2$. Thus $R$ is a Gorenstein ring. We want to show $\rme(R) \le 2$. Assume that $\rme(R) \ge 2$ and consider $B = \m :\m$. Then, $B \in \calY_R$ and $R \ne B$, because $R$ is not a DVR. Consequently, because $\m = R : B$, by Lemma \ref{5.1}  $\m^2 = a\m$ for some $a \in \m$, which implies $\rme(R) = 2$, since $R$ is a Gorenstein ring.
\end{proof}

We say that a Noetherian ring $R$ satisfies the condition ${\rm (S_1)}$ of Serre, if $\depth R_\fkp \ge \inf\{1, \dim R_\fkp \}$ for every $\fkp \in \Spec R$.

\begin{cor}\label{5.3}
Let $R$ be a Noetherian ring and suppose that $R$ satisfies ${\rm (S_1)}$. Then, $\rme(R_\fkp) \le 2$ for every $\fkp \in \operatorname{Ht_1}(R)$, if $R$ is an anti-stable ring. 
\end{cor}

\begin{proof}
Let $\fkp \in \operatorname{Ht_1}(R)$ and set $A = R_\fkp$. Hence $A$ is a Cohen-Macaulay local ring of dimension one. Let $I \in \calF_A$ and set $J = I \cap R$. We will show that $A:I$ is a cyclic $(I:I)$-module. We may assume that $I \ne A$. Hence, $J$ is a $\fkp$-primary ideal of $R$, since $I$ is a $\fkp R_\fkp$-primary ideal of $A =R_\fkp$. Hence, because $J \in \calF_R$ (remember that $R$ satisfies ${\rm (S}_1$$)$), $R:J$ is a projective $(J:J)$-module. Therefore, $A:I = [R:J]_\fkp$ is a cyclic module over $I:I = [J:J]_\fkp$, since it has rank one over the semi-local ring $I:I$. Thus, $\rme(A) \le 2$ by Theorem \ref{5.6}.
\end{proof}

We now come to the main results of this section.

\begin{thm}\label{5.11}
Let $R$ be a Noetherian ring and suppose that $R$ satisfies ${\rm (S_1)}$. Let us consider the following four conditions.
\begin{enumerate}[{\rm $(1)$}]
\item $R$ is anti-stable.
\item $R$ is strongly anti-stable.
\item Every $I \in \calF_R$ is isomorphic to $\tau_R(I)$.
\item $\rme(R_\fkp) \le 2$ for every $\fkp \in \operatorname{Ht_1}(R)$.
\end{enumerate}
Then, we have the implications $(3) \Leftrightarrow (2) \Rightarrow (1) \Rightarrow (4)$. If $R$ is semi-local $($resp. $\dim R=1$$)$, then the implication $(1) \Rightarrow (2)$ $($resp. $(4) \Rightarrow (1)$$)$ holds true.
\end{thm}

\begin{proof}
$(3) \Leftrightarrow (2) \Rightarrow (1) \Rightarrow (4)$ See Lemma \ref{5.0} and  Theorem \ref{5.3}.

If $R$ is semi-local, then every birational module -finite extension of $R$ is also semi-local, so that the implication $(1) \Rightarrow (2)$ follows.

Suppose that $\dim R = 1$. Let $I \in \calF_R$ and set $A = I:I$. Then, by Theorem \ref{5.6} $R_M:IR_M=[R:I]_M$ is a cyclic $A_M$-module for every $M \in \Max R$, so that $R:I$ is an invertible $A$-module. Hence, the implication $(4) \Rightarrow (1)$ follows.
\end{proof}


\begin{thm}\label{final}
Let $R$ be a Cohen-Macaulay ring with $\dim R_M = 1$ for every $M \in \Max R$. If $R$ is an anti-stable ring, then $R$ is a stable ring. 
\end{thm}

\begin{proof}
For every $M \in \Max R$, $\rme(R_M) \le 2$ by Corollary \ref{5.3}. Let $I$ be an arbitrary ideal of $R$ and set $A = \End_RI$. Then, because $R_M$ is a stable ring by Proposition \ref{5.2a}, for every $M \in \Max R$ $IR_M$ is a projective $A_M$-module, so that $I$ is a projective $A$-module. Thus, $R$ is a stable ring.
\end{proof}

\begin{acknowledgement} {\rm The authors are grateful to M. Hashimoto and K. Kurano. Examples \ref{hashimoto} and \ref{not self-injective} are deep in debt of them. Example \ref{hashimoto} was pointed by Hashimoto, and Kurano informed the authors about the paper \cite{F}. The authors are also very grateful to R. Takahashi for his detailed comments and kind suggestions about the manuscript, which he gave the authors while they were preparing for the final version. }
\end{acknowledgement}


\end{document}